%% file: p1_II.tex
\documentclass[12pt]{amsart}
\usepackage{fullpage,xcolor,graphicx}
\usepackage{hyperref}
\usepackage{pifont}
\usepackage{amsthm}
\usepackage{amsfonts}
\usepackage{amssymb}
\usepackage[mathscr]{euscript}
\usepackage[all]{xy}
\usepackage{amsmath}
\usepackage{epsfig}

\newtheorem{theorem}{Theorem}
\newtheorem{lemma}[theorem]{Lemma}

\newtheorem{corollary}[theorem]{Corollary}

\theoremstyle{definition}

\theoremstyle{remark}

\begin{document}
\title{Virtual Covers of Links}

\author{Micah Chrisman}

\begin{abstract} We use virtual knot theory to detect the non-invertibility of some classical links in $\mathbb{S}^3$.  These links appear in the study of virtual covers. Briefly, a virtual cover associates a virtual knot $\upsilon$ to a knot $K$ in a $3$-manifold $N$, under certain hypotheses on $K$ and $N$. Virtual covers of links in $\mathbb{S}^3$ come from taking $K$ to be in the complement $N$ of a fibered link $J$. If $J \sqcup K$ is invertible and $K$ is ``close to'' a fiber of $J$, then $\upsilon$ satisfies a symmetry condition to which some virtual knot polynomials are sensitive. We also discuss virtual covers of links $J \sqcup K$ where $J$ is not fibered but is virtually fibered (in the sense of W. Thurston).   
\end{abstract}
\keywords{virtual knots, non-invertible links, Vassiliev filtration, virtually fibered links}
\subjclass[2010]{57M25, 57M27}
\maketitle

It is well-known that quantum knot invariants cannot distinguish a knot in $\mathbb{S}^3$ from its inverse. For links, it is unknown if quantum invariants detect non-invertibility (see Duzhin-Karev \cite{dk}). The HOMFLYPT polynomial in particular does not detect it. Other methods for knots (e.g. Hartley \cite{hartley}) and links (e.g. Whitten \cite{whitten_sublinks}) can be used, but they are difficult. The situation for virtual knots is better: easily computable asymmetry sensitive virtual knot invariants exist. Morever, Manturov and the author recently \cite{cm_fiber} developed \emph{virtual covers}. This technique associates virtual knots to \emph{some} two component links. Here we exploit these facts to give simple proofs of non-invertibility for \emph{some} links in $\mathbb{S}^3$ (see Section \ref{sec_non}). 
\newline
\newline
To do this, we build on \cite{cm_fiber} and establish a condition on links in $\mathbb{S}^3$ so that the associated virtual knot functions essentially as a link invariant (see Section \ref{sec_main}). Besides detecting non-invertibility, the result provides an avenue for studying geometric properties of classical links with virtual knots. Further applications will be considered in Chrisman-Kaestner \cite{vc_2}. 
\newline
\newline
It is thus desirable to extend virtual covers from \emph{some} links to \emph{some more} links.  In \cite{cm_fiber}, two component links of the form $J \sqcup K$ with $J$ fibered and $\text{lk}(J,K)=0$ were considered. This generalizes easily to $m+1$ component links $J \sqcup K$ with $J$ a fibered $m$ component link and $K$ ``close to'' a fiber of $J$ (see Section \ref{sec_ssf}). Section \ref{sec_virt_fib} considers the case that $J$ is not necessarily fibered but that $\mathbb{S}^3 \backslash J$ has a finite-index cover by $\mathbb{S}^3\backslash \tilde{J}$, where $\tilde{J}$ is a fibered link. Such links $J$ are examples of \emph{virtually fibered links}, a class of links that conjecturally contains all hyperbolic links (see Walsh \cite{walsh}).
\section{Virtual Covers and Non-Invertibility of Links}
 
\subsection{Review: Virtual Knots} The four models of virtual knots are:
\begin{enumerate}
\item virtual knot diagrams modulo the extended Reidemeister moves (Kauffman \cite{KaV}),  
\item Gauss diagrams modulo diagram Reidemeister moves (Goussarov-Polyak-Viro \cite{GPV}), 
\item abstract knot diagrams modulo Kamada--Kamada equivalence \cite{kamkam}, and
\item knots in thickened surfaces modulo stabilization/destabilization (Kuperberg \cite{kuperberg}). 
\end{enumerate}
There is a one-to-one correspondence between the equivalence classes of each interpretation. If we have a representative $R$ of one of the models (2), (3), or (4), we will denote by $\kappa(R)$ the corresponding equivalence class of virtual knot diagrams from model (1).  The four models for a single virtual knot are shown below.

\begin{figure}[htb]
\fcolorbox{black}{white}{
$
\begin{array}{cccc}
\begin{array}{c}
\scalebox{.6}{\psfig{figure=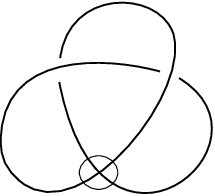}} \end{array} &  \begin{array}{c}
\scalebox{.6}{\psfig{figure=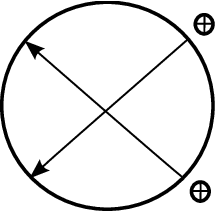}} \end{array} & 
\begin{array}{c}
\scalebox{.6}{\psfig{figure=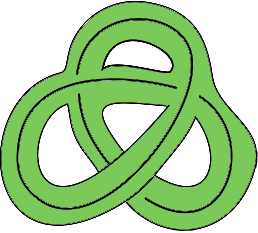}}
\end{array} &
\begin{array}{c}
\scalebox{.6}{\psfig{figure=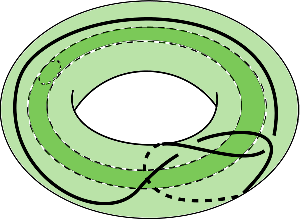}}
\end{array} \\
(1) & (2) & (3) & (4) \\ 
\end{array}
$}
\end{figure}

We use ``$\leftrightharpoons$'' to denote equivalence of virtual knots. If a virtual knot has a diagram whose crossings are all classical, then it is said to be a \emph{classical knot}. If two classical knots are equivalent as virtual knots, then they are equivalent as knots in $\mathbb{S}^3$ \cite{GPV}. A virtual knot is classical if and only if it can be represented as a diagram on a compact surface in $\mathbb{R}^2$.

\subsection{Virtual Covers} Let $M$ be a smooth connected oriented $3$-manifold and $K^M$ a smooth oriented knot in $\text{int}(M)$. Knots in $M$ are considered equivalent up to ambient isotopy in $M$. If $\partial M \ne \emptyset$, we assume that the isotopy acts as the identity on $\partial M$. We write $K_1^M \leftrightharpoons K_2^M$ if $K_1^M$ and $K_2^M$ are equivalent in $M$.
\newline
\newline
Let $N$ be a compact connected oriented (c.c.o) $3$-manifold having a regular orientation preserving covering space $\Pi:\Sigma \times \mathbb{R} \to N$, where $\Sigma$ is a c.c.o smooth $2$-manifold. If $K^N$ is a knot in $N$ and $\mathfrak{k}^{\Sigma \times \mathbb{R}}$ is a knot in $\Sigma \times \mathbb{R}$ satisfying $\Pi(\mathfrak{k})=K$, we call $\mathfrak{k}^{\Sigma \times \mathbb{R}}$ a \emph{lift by} $\Pi$ of $K^N$. A \emph{virtual cover} of $K^N$ is a choice $\mathfrak{k}^{\Sigma \times \mathbb{R}}$ of lift by $\Pi$ for $K^N$, denoted $(\mathfrak{k}^{\Sigma \times \mathbb{R}},\Pi,K^N)$. The oriented virtual knot $\upsilon=\kappa(\mathfrak{k}^{\Sigma \times \mathbb{R}})$ is called the \emph{associated virtual knot}.

\begin{figure}[htb]
\fcolorbox{black}{white}{
$\xymatrix{\begin{array}{c}
\scalebox{.4}{\psfig{figure=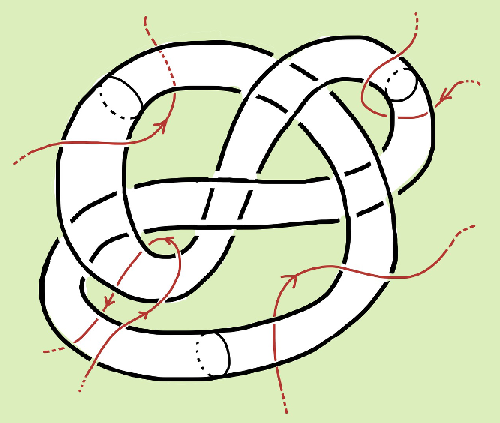}} \end{array} & \ar[l]^{\Pi} \ar@{-->}[r]^{\kappa(\mathfrak{k})=\upsilon} \begin{array}{c}
\scalebox{.45}{\psfig{figure=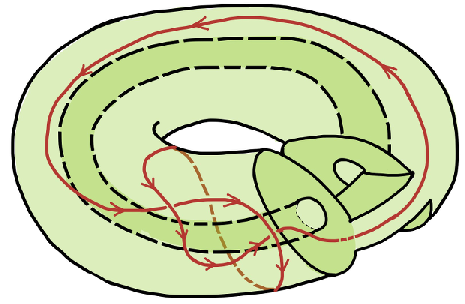}} \end{array} & \begin{array}{c}
\scalebox{.5}{\psfig{figure=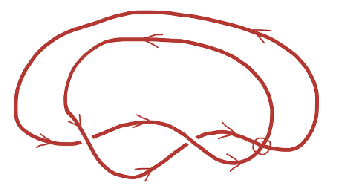}}
\end{array} 
}
$}
\end{figure}

Viewing $K$ as a closed path at a point $x_0$ in $N$, regularity implies there is a one-to-one correspondence between lifts by $\Pi$ and elements of $\Pi^{-1}(x_0)$. If $K^N$ has a lift by $\Pi$, let $\Upsilon(K^N)$ be the set of associated virtual knots for all lifts by $\Pi$ of $K^N$. If $|\Upsilon(K^N)|=1$, i.e. all the associated virtual knots are equivalent, then the unique element $\upsilon$ is called the \emph{invariant associated virtual knot}. In our notation, the main results of \cite{cm_fiber} can be restated compactly as follows (see also Krasnov-Manturov \cite{km_fiber2}).

\begin{lemma} \label{lemma_cm_general} Let $(\mathfrak{k}_0^{\Sigma \times \mathbb{R}},\Pi,K_0^N),(\mathfrak{k}_1^{\Sigma \times \mathbb{R}},\Pi,K_1^N)$ be virtual covers with invariant associated virtual knots $\upsilon_0,\upsilon_1$, respectively. If $K_0^N \leftrightharpoons K_1^N$, then $\upsilon_0 \leftrightharpoons \upsilon_1$.
\end{lemma}
\begin{proof} This follows from the definitions and the proof of Lemma 3 from \cite{cm_fiber}.
\end{proof}

\subsection{Interlude on Fibered Links} \label{section_interlude} Let $V(X)$ denote some sufficiently small tubular neighborhood of $X$. For a link $J$ in $\mathbb{S}^3$, let, $N_J=\overline{\mathbb{S}^3 \backslash V(J)}$. By a fibered link, we will mean an oriented link in $\mathbb{S}^3$ having a c.c.o Seifert surface $\Sigma_J$ such that the pair  $(\overline{\mathbb{S}^3\backslash V(\Sigma_J)},\overline{\mathbb{S}^3\backslash V(\Sigma_J)} \cap \partial N_J)$ is diffeomorphic as a pair to $(\Sigma_J \cap N_J, \partial (\Sigma_J \cap N_J)) \times \mathbb{I}$ (Kawauchi \cite{kawauchi}). The surface $\Sigma_J$, called a \emph{fiber}, must be oriented so that the induced orientation on the boundary is the orientation of $J$. Note we may identify $N_J$ with a mapping torus $(\Sigma_J \cap N_J)\times \mathbb{I}/\psi$, with $\psi: \Sigma_J \cap N_J \to \Sigma_J \cap N_J$ an orientation preserving (o.p.) diffeomorphism. This defines a regular covering space $\Pi_J: (\Sigma_J \cap N_J) \times \mathbb{R} \to N_J$. \footnote{Henceforth, we will frequently abuse notation and fail to distinguish between $\Sigma_J$ and $\Sigma_J \cap N_J$.} 
\newline
\newline
The Seifert surface $\Sigma_J$ is of minimal genus and incompressible \cite{kawauchi}. It is also unique in the sense that for any two minimal genus c.c.o Seifert surfaces $\Sigma_1,\Sigma_2$ of $J$, there is an ambient isotopy of $\mathbb{S}^3$ fixing $J$ and taking $\Sigma_1$ to $\Sigma_2$ (see Kobayshi \cite{kobayashi_minimal} and Whitten \cite{whitten}). This is a special property of fibered links and is not always the case for an arbitrarily chosen link.
\newline
\newline
If two fibered links have known fibers, then a new fibered link can be constructed using a generalization of the boundary connect sum called \emph{Murasugi's sums} \cite{gabai_nat}. The resulting link is fibered and the constructed surface is a fiber.  Other techniques for finding fibers $\Sigma_J$ are described in \cite{stallings}. More recently, there is also \emph{Futer's method} \cite{futer}. We will take the literature on the topic as a calculus we employ to find fibers. To manipulate fibers, draw them \emph{disc-band form} (Burde-Zieschang \cite{bz}) and isotope using \emph{topological script} (Kauffman \cite{on_knots}).

\subsection{Examples of Virtual Covers} \label{sec_ssf} Consider an $m+1$ component oriented link $L=J \sqcup K$ in $\mathbb{S}^3$, where $J$ is a fibered $m$ component link and $K$ is a knot in the complement $N_J$. The components of $L$ are ordered, where $K$ is the ``last'' component. As  above, $\Sigma_J$ is an oriented c.c.o Seifert surface for $J$. It is two-sided in $\mathbb{S}^3$, denoted red/blue in figures. The blue side is always chosen so that resting a thumb-up right hand on it reproduces the orientation of the surface and the right-handed orientation of $\mathbb{S}^3$.
\newline
\newline
If $K$ lies in a small neighborhood of $\Sigma_J$, a lift by $\Pi_J$ is guaranteed. More precisely, we say that $K^{N_J}$ is in \emph{special Seifert form (SSF) relative to} $\Sigma_J$ if $K^{N_J}$ can be decomposed into a finite number of disjoint arcs in $\Sigma_J$ and a finite number of ``crossings'' in small disjoint embedded $3$-balls $B_i \subset \text{int}(N_J)$: each $B_i \cap \Sigma_J$ is an embedded disc dividing $B_i$ into halves, so that $K \cap B_i$ consists of two disjoint arcs in different halves of $\partial B_i$ configured as in the model on the left side of Figure \ref{fig_virt_cov_ex} (see also \cite{cm_fiber}). As $\Pi_J$ is a local diffeomorphism and $\Pi_J^{-1}(\Sigma_J) \approx \Sigma_J \times \mathbb{Z}$,  any $K^{N_J}$ in SSF relative to $\Sigma_J$ has a lift by $\Pi_J$. Hence we have virtual covers of the form $(\mathfrak{k}^{\Sigma_J \times \mathbb{R}},\Pi_J,K^{N_J})$.  
\newline

\begin{figure}[htb]
\begin{tabular}{|c|cc|} \hline & & \\
\begin{tabular}{c}
\def\svgwidth{2.5in}
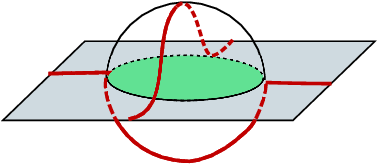 \end{tabular} 
 &
\begin{tabular}{c}
\def\svgwidth{1in}
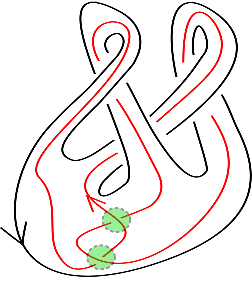 \end{tabular} & \begin{tabular}{c} \def\svgwidth{.9in}
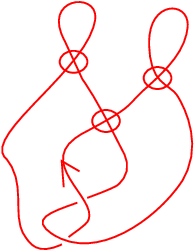 \end{tabular} \\ \hline
\end{tabular}
\caption{(Left) A ``crossing'' in special Seifert form. (Right) A knot $K$ in SSF wrt a left handed trefoil fiber and the invariant associated virtual knot $\upsilon$.} \label{fig_virt_cov_ex}
\end{figure}

If $K^{N_J}$ is in SSF relative to $\Sigma_J$, it has an invariant associated virtual knot. Note that $\Sigma_J$ is a framed neat submanifold of $N_J$ (Kosinski \cite{kosinski}). Orient the normal bundle by choosing a normal vector $\vec{v}$ at each point of $\Sigma_J$ so that the surface orientation together with $\vec{v}$ gives the right-handed orientation of $\mathbb{S}^3$. Assume the SSF lies in this oriented tubular neighborhood $\Sigma_J \times (-1,1)$ where $\Sigma_J\equiv\Sigma_J \times \{0\}$. An arc in a ball of the SSF is \emph{overcrossing} (resp. \emph{undercrossing}) if it lies in $\Sigma_J \times [0,1)$ (resp. $\Sigma_J \times (-1,0])$. The SSF thus defines a knot diagram $[K;\Sigma_J]$ on $\Sigma_J$. By \cite{cm_fiber}, the invariant associated virtual is $\kappa([K;\Sigma_J])$. See Figure \ref{fig_virt_cov_ex}.

For an example of virtual covers of knots $K^N$ where $N$ is closed, take $L=J \sqcup K$ where $J$ is a fibered knot in $\mathbb{S}^3$ and $K$ is in SSF relative to $\Sigma_J$. Let $N$ be the manifold obtained by performing the $0$-surgery on $J$. Then $N$ is a fibered $3$-manifold. A fiber can be obtained by gluing a disc to $\partial \Sigma_J \approx \mathbb{S}^1$. Hence $K^N$ has an invariant associated virtual knot as above.

\subsection{The Main Theorem} \label{sec_main} We show that the associated virtual knot can be used essentially as a link invariant. It is not an invariant in the traditional sense, but it is an invariant in the sense that if two equivalent links of the form $L=J \sqcup K$ as above have associated virtual knots, then the associated virtual knots must be equivalent. First, a lemma is needed.

\begin{lemma} \label{lemma_inheritance} If $K$ is in SSF with respect to $\Sigma_J$ and $H:\mathbb{S}^3 \to \mathbb{S}^3$ is an o.p. diffeomorphism, then $H(K)$ is in SSF with respect to $H(\Sigma_J)$ and $\kappa([K;\Sigma_J])\leftrightharpoons\kappa([H(K);H(\Sigma_J)])$. 
\end{lemma}
\begin{proof} An SSF is defined by embedded $3$-balls $B_i$, discs on $\Sigma_J$, arcs on $\Sigma_J$, and ``crossings'' in balls $B_i$. These are mapped by $H$ to $3$-balls, disc/arcs on $H(\Sigma_J)$ and ``crossings'' in balls $H(B_i)$.  $H(\Sigma_J)$ inherits its orientation from $\Sigma_J$ and is a canonically oriented framed neat submanifold of $\overline{\mathbb{S}^3\backslash V(\partial H(\Sigma_J))}$. $H$ is o.p., so the overcrossing arcs of the SSF for $K$ are mapped to the canonically chosen overcrossing arcs for $H(K)$ in SSF with respect to $H(\Sigma_J)$. Observe that the Gauss diagrams of $[K;\Sigma_J]$ and $[H(K);H(\Sigma_J)]$ are equivalent. 
\end{proof}

\begin{theorem}[Main Theorem] \label{theorem_link_invar} For $i=1,2$, let $L_i=J_i \sqcup K_i$ be oriented links with $J_i$ a fibered link, $K_i$ in SSF with respect to a fiber $\Sigma_i$ of $J_i$, and $\upsilon_i$ the invariant associated virtual knot. If $L_1 \leftrightharpoons L_2$, then $\upsilon_1 \leftrightharpoons \upsilon_2$.
\end{theorem}
\begin{proof} By hypothesis, there is an o.p. diffeomorphism $H:\mathbb{S}^3 \to \mathbb{S}^3$ taking $J_1$ to $J_2$ and $K_1$ to $K_2$. Let $\Sigma_2'=H^{-1}(\Sigma_2)$. This is a minimal genus c.c.o Seifert surface for $J_1$. Since $J_1$ is a fibered link, its minimal genus Seifert surfaces are unique. Hence, there is an ambient isotopy $G$ of $\mathbb{S}^3$ taking $\Sigma_2'$ to $\Sigma_1$ and fixing $J_1$ pointwise. In particular, the inherited orientation on $\Sigma_2'$ maps to the orientation of $\Sigma_1$. By Lemma \ref{lemma_inheritance}, $K_1=H^{-1}(K_2)$ is in SSF with respect to $\Sigma_2'=H^{-1}(\Sigma_2)$ and $\kappa ([K_1;\Sigma_2']) \leftrightharpoons \upsilon_2$. Note also that $G$ takes $K_1$ to a knot $K_1'$ in SSF with respect to $\Sigma_1$. Again by Lemma \ref{lemma_inheritance}, it follows that $\upsilon_2 \leftrightharpoons \kappa([K_1';\Sigma_1])$.
\newline
\newline
The ambient isotopy $G$ is also an isotopy of the embeddings of circles $K_1$ and $K_1'$. Since $K_1$ and $K_1'$ are in SSF with respect to $\Sigma_2'$ and $\Sigma_1$, respectively, each level of the isotopy must lie in the interior of $N_{J_1}$. Hence, by the isotopy extension theorem \cite{kosinski}, there is an ambient isotopy of $N_{J_1}$ taking $(K_1)^{N_{J_1}}$ to $(K_1')^{N_{J_1}}$. Thus there are virtual covers $(\mathfrak{k}_1^{\Sigma_1 \times \mathbb{R}},\Pi,K_1^{N_{J_1}})$ and $((\mathfrak{k}_1')^{\Sigma_1 \times \mathbb{R}},\Pi,(K_1')^{N_{J_1}})$ and $(K_1)^{N_{J_1}}\leftrightharpoons (K_1')^{N_{J_1}}$. By Lemma \ref{lemma_cm_general}, it follows that $\upsilon_1 \leftrightharpoons \kappa([K_1;\Sigma_1]) \leftrightharpoons \kappa([K_1';\Sigma_1]) \leftrightharpoons \upsilon_2$.
\end{proof}

An immediate corollary is the following strengthening of Theorem 2 from \cite{cm_fiber}. It will be used to show that some links are non-split.

\begin{corollary}\label{cor_split} Let $L=J \sqcup K$ with $J$ a fibered link, $K$ in SSF with respect to a fiber $\Sigma_J$ of $J$, and $\upsilon$ the invariant associated virtual knot. If there is an embedded $3$-ball $B$ in $\mathbb{S}^3$ such that $K \subseteq \text{int}(B)$ and $J \subseteq \mathbb{S}^3 \backslash B$ (so that $L$ is split), then $\upsilon \leftrightharpoons K$.
\end{corollary}
\begin{proof} There is an ambient isotopy moving $K$ far away from $\Sigma_J$. Then $K$ can be drawn in SSF with respect to some embedded $2$-disc $D$ disjoint from $\Sigma_J$. This disc can be made arbitrarily small and moved back to a disc $D'$ on $\Sigma_J$. $K'$ is in SSF with respect to $\Sigma_J$ and clearly its invariant associated virtual $\upsilon' \leftrightharpoons K$. By Theorem 3, $\upsilon \leftrightharpoons K$.
\end{proof}

\textbf{Example:} The three component link $L$ on the right in Figure \ref{fig_hopf_link_virt_cover} is not split. Suppose it is split, so that $\mathbb{S}^3$ decomposes into $3$-balls $B_1 \cup B_2$ each containing a component of $L$. The Hopf link $J$ must have its components in the same ball, say $B_1$, since it has linking number one. The third component $K$ cannot be contained in $B_2$: $K$ is in SSF with respect to a fiber of the Hopf link with trivial invariant associated virtual knot and $K^{\mathbb{S}^3} \leftrightharpoons 4_1$ isn't trivial. \hfill $\square$

\begin{figure}[htb]
\centerline{
\fcolorbox{black}{white}{
\begin{tabular}{ccccc}
\def\svgwidth{1.25in}
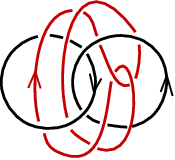 
& \hspace{.5cm} &
\def\svgwidth{2in}
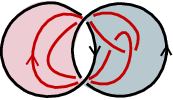 &\hspace{.5cm} & \def\svgwidth{1.5in}
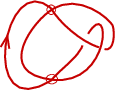
\end{tabular}
}
}
\caption{(Left) A non-split link $J \sqcup K$. (Middle) $K$ in SSF with respect to a Hopf link fiber. (Right) The invariant associated virtual knot is trivial.}\label{fig_hopf_link_virt_cover}
\end{figure}

\subsection{Non-Invertible Links} \label{sec_non} This section applies Theorem \ref{theorem_link_invar} to the detection of non-invertible links in $\mathbb{S}^3$. Denote as usual the change of orientation of all components of $L$ by $-L$. An oriented link $L$ is said to be \emph{invertible} if $L \leftrightharpoons -L$ (Note: the given ordering of the components is not changed). 
\newline
\newline
If $\mathfrak{k}$ is a knot diagram on a c.c.o surface, let $D_{\mathfrak{k}}$ denote any Gauss diagram of $\mathfrak{k}$. For any Gauss diagram $D$, let $\mathfrak{f}(D)$ denote the diagram obtained from $D$ by changing the direction of each arrow (but not the sign).

\begin{theorem} Let $L=J \sqcup K$ be an $m+1$ component oriented link, with $J$ an $m$ component fibered link, $\Sigma_J$ a fiber, and $K$ a knot in SSF with respect to $\Sigma_J$. If $L=J \sqcup K$ is invertible, then $\kappa(D_{[K;\Sigma_J]}) \leftrightharpoons -\kappa(\mathfrak{f}(D_{[K;\Sigma_J]}))$.
\end{theorem}
\begin{proof} Recall that the orientation of $\Sigma_J$ induces the given orientation on $\partial \Sigma_J=J$. A fiber $\Sigma_J'$ for $-J$ may be obtained by changing the orientation of $\Sigma_J$. This also changes the framing of $\Sigma_J$; just switch the blue and red sides. Clearly, $-K$ is in SSF with respect to $\Sigma_J'$. A Gauss diagram for $[-K;\Sigma_J']$ can be obtained from $D_{[K;\Sigma_J]}$ by changing the orientation of $D_{[K;\Sigma_J]}$ and switching each overcrossing arc in the SSF to an undercrossing arc (and vice versa). This does not change the local orientation of the crossing. Now apply Theorem \ref{theorem_link_invar}.
\end{proof}

\textbf{Example:} The link $L=J \sqcup K$ on the left in Figure \ref{figure_non_invert} is a non-split non-invertible link with invertible components. An evident Seifert surface $\Sigma_J$ in disc-band form is depicted. $J$ is fibered and invertible: it is a connect sum of $3_1$ (leftmost two bands), $5_1$ (next four bands) and $3_1$(rightmost two bands).  As a Murasugi sum of fibers of fibered knots, $\Sigma_J$ is a fiber. $K$ is invertible as it  a left handed trefoil (after some effort). A Gauss diagram $D$ of the invariant associated virtual knot $\upsilon$ is on the right of the figure. The (reduced) Sawollek polynomial \cite{saw} shows that $D \not \leftrightharpoons -\mathfrak{f}(D)$: 
\begin{eqnarray*}
\tilde{Z}_D(x,y) &=& -1+x^2-x^3+x^4+\frac{x^3}{y^3}+\frac{x^4}{y^3}-y+x y+xy^2-x^2y^2,\\
\tilde{Z}_{-\mathfrak{f}(D)}(x,y) &=& -1+x^2-x^3+x^4+\frac{x^3}{y^2}-\frac{x^4}{y^2}-\frac{x}{y}+\frac{x^2}{y}-y^3+xy^3.\\
\end{eqnarray*}
Hence, $L$ is not invertible. The Sawollek polynomial of a classical knot is $0$, so $\upsilon$ is non-classical and hence not equivalent to $K$. By Corollary \ref{cor_split}, $L$ is non-split. \hfill $\square$
\newline

\begin{figure}[htb]
\fcolorbox{black}{white}{
\begin{tabular}{cc}
\begin{tabular}{c}
\def\svgwidth{3.35in}
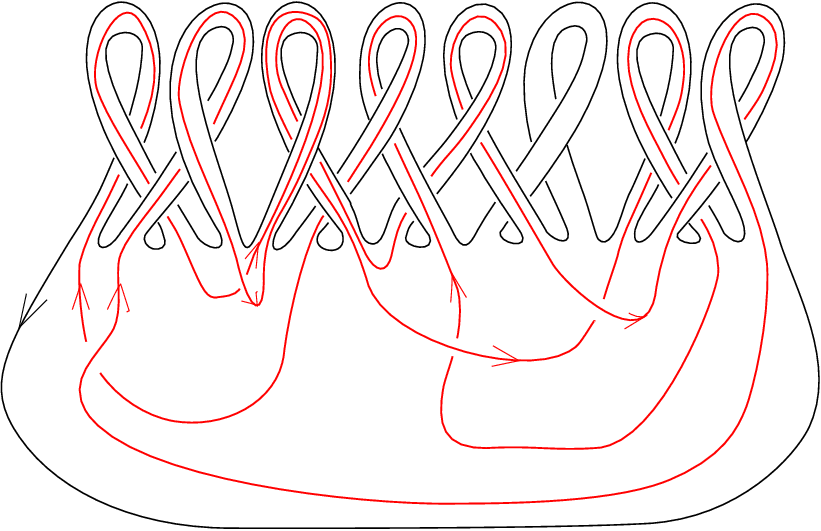 \end{tabular} &
\begin{tabular}{c}
\def\svgwidth{1.25in}
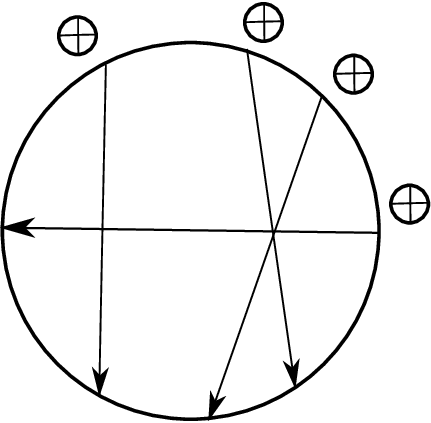 \end{tabular}
\end{tabular}}
\caption{(Left) A non-invertible non-split link with invertible components. (Right) The invariant associated virtual knot.}\label{figure_non_invert}
\end{figure}

The previous example can be generalized to any number of components. Whitten proved \cite{whitten_sublinks} that for every $m \ge 2$, there is an $m$ component non-invertible non-split link all of whose proper sublinks are invertible. This is the inspiration behind the following result.

\begin{corollary}\label{corollary_whitten_almost} For every $m \ge 2$, there are infinitely many $m$ component non-invertible non-split links $L_{m,i}$ all of whose components are invertible. Moreover, $L_{m,i}=J_{m,i,1} \sqcup \cdots \sqcup J_{m,i,m-1} \sqcup K_{m,i}$ may be chosen so that each proper sublink not containing $J_{m,i,1}$ is invertible. 
\end{corollary}
\begin{proof} Let $\Sigma_{2,1}$ be the surface indicated in Figure \ref{figure_non_invert}. Construct a surface $\Sigma_{m,i}$ by taking Murasugi sums \cite{gabai_nat} with $i-1$ trefoil fibers $\Sigma_T$ and $m-2$ Hopf link fibers $\Sigma_H$. The desired operations are topologically boundary connect sums of surfaces, as indicated in Figure \ref{figure_whitten}.  $\Sigma_{m,i}$ is an oriented fiber of the oriented fibered link $\partial \Sigma_{m,i}$ (Stallings \cite{stallings}). The link $\partial \Sigma_{m,i}$ has $m-2$ unknotted components and one component that is a connect sum of $5_1$ and $i+1$ right handed trefoils. Set this last component to be $J_{m,i,1}$ and label the other components arbitrarily. Define $K_{m,i}$ to be the (red) knot $K$ from Figure \ref{figure_non_invert}.  
\newline
\newline
By construction, $K_{m,i}$ is in SSF with respect to $\Sigma_{m,i}$. Its invariant associated virtual knot is depicted on the right in Figure \ref{figure_non_invert}. Thus $L_{m,i}$ is non-invertible (see previous example). All of the components are invertible. Any proper sublink not containing $J_{m,i,1}$ is a split link that splits into a trefoil and an unlink. Such a link is invertible.
\newline
\newline
Lastly we must show that $L_{m,i}$ is non-split. If it is split, then $\mathbb{S}^3$ can be decomposed into two $3$-balls $B_1$ and $B_2$ each containing at least one component of $L_{m,i}$. Linking number considerations imply that all of the components $J_{j,i,k}$ must be in the same 3-ball, say $B_1$. But $K_{j,i}$ cannot be in $B_2$ since the associated virtual knot is not classical (by Corollary \ref{cor_split}). 
\end{proof} 

\begin{figure}[htb]
\fcolorbox{black}{white}{
\begin{tabular}{c}
\def\svgwidth{6in}
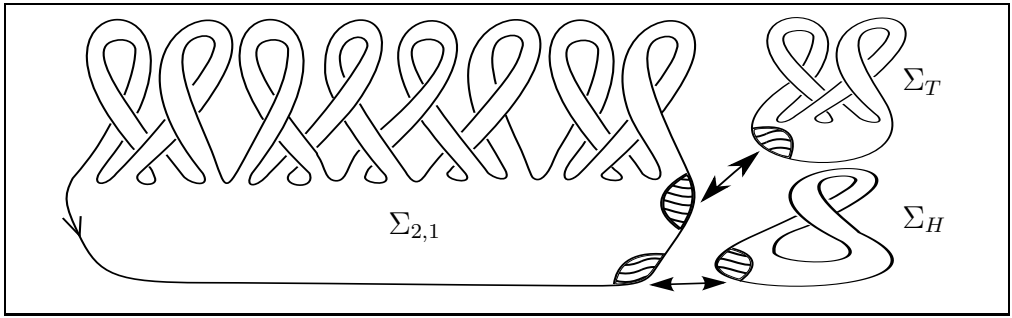
\end{tabular}}
\caption{Murasugi sums used in the proof of Corollary \ref{corollary_whitten_almost}.}\label{figure_whitten}
\end{figure}
  
\section{Virtual Covers and Virtually Fibered Links} \label{sec_virt_fib}

Previously links in $\mathbb{S}^3$ of the form $J \sqcup K$ where $J$ is fibered have been considered. Virtual covers can also be applied to links of the form $J \sqcup K$ where $J$ is not fibered. We will consider links where $N_J$ has a finite index cover by $N_{\tilde{J}}$, where $\tilde{J}$ is a fibered link. Such a link is an example of a \emph{virtually}\footnote{The double use of the word ``virtual'' is admittedly confusing, but unavoidable.} \emph{fibered link} \cite{walsh}. More generally, a virtually fibered $3$-manifold is one which has a finite index cover by a fibered $3$-manifold (Thurston \cite{thurston}). 
\newline
\begin{figure}[htb]
\[
\fcolorbox{black}{white}{
\xymatrix{
\begin{array}{c}
\def\svgwidth{2.5in}
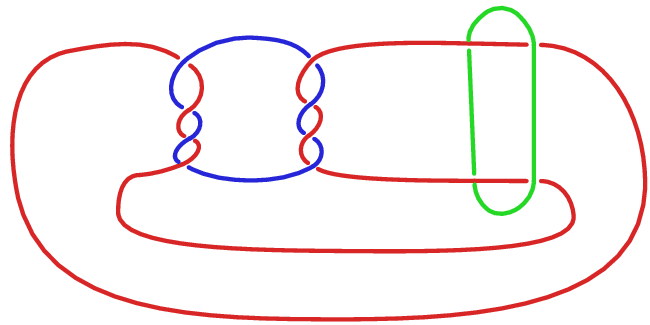 
\end{array}
\ar[r] &
\begin{array}{c}
\def\svgwidth{2in}
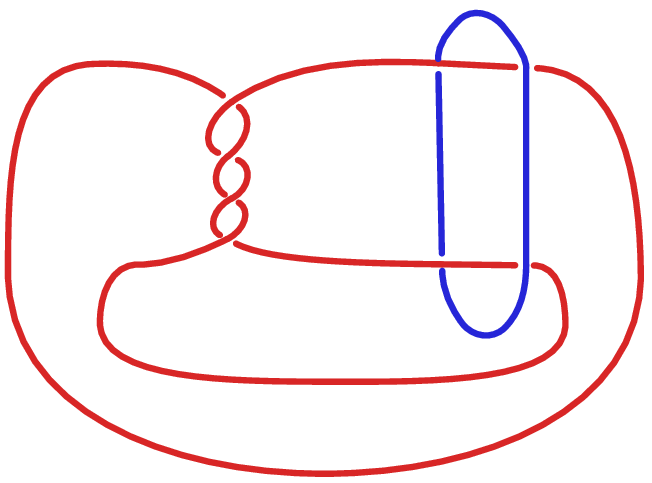
\end{array}
}
}
\]
\caption{The links $\tilde{J}$ (Left) and $J$ (Right) (drawn in SnapPy).}\label{fig_gabai_links}
\end{figure}

\textbf{Example:} We review Gabai's \cite{gabai_first} example of a non-fibered, virtually fibered hyperbolic link. Let $J$ be the two component link on the right hand side of Figure \ref{fig_gabai_links} and $\tilde{J}$ the three component link on the left hand side of Figure \ref{fig_gabai_links}. An index two covering projection $N_{\tilde{J}} \to N_J$ can be described as follows. A ``dented cube'' $C$ with a tangle removed is depicted at the bottom of Figure \ref{fig_gabai_cubes}. This is a fundamental region of an index two covering projection $C \sqcup C \to C$ (see Figure \ref{fig_gabai_cubes}). $N_J$ is constructed by identifying the golden annulus $Y$ and red annulus $R$ in Figure \ref{fig_gabai_cubes_identifications} and then identifying the blue $(B)$ and green $(G)$ twice punctured discs in Figure \ref{fig_gabai_cubes_identifications}.  The corresponding identifications on $C \sqcup C$ produce $N_{\tilde{J}}$. The calculus described in Section \ref{section_interlude} can be used to show that the surface $\Sigma_{\tilde{J}}$ in Figure \ref{fig_gabai_fiber}, after an ambient isotopy of $\tilde{J}$, is a fiber for $\tilde{J}$. By \cite{gabai_first}, $J$ is non-fibered and hyperbolic. \hfill $\square$

\begin{figure}[htb]
\[
\fcolorbox{black}{white}{
\xymatrix{
\begin{array}{c} 
\scalebox{.75}{\psfig{figure=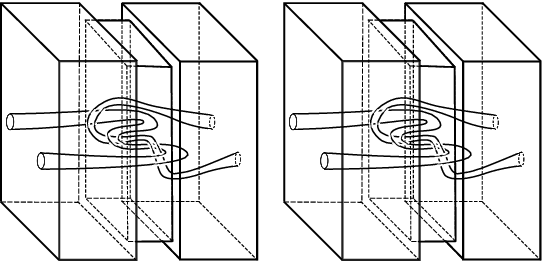}} \end{array} \ar[d] \\  
\begin{array}{c} 
\scalebox{.75}{\psfig{figure=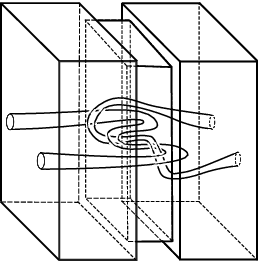}} \end{array}
}
}
\]
\caption{Constructing Gabai's double cover $N_{\tilde{J}}$ of $N_{J}$.} \label{fig_gabai_cubes}
\end{figure}

\begin{figure}[htb]
\[
\fcolorbox{black}{white}{
\xymatrix{
\begin{array}{c}
\def\svgwidth{1.25in}
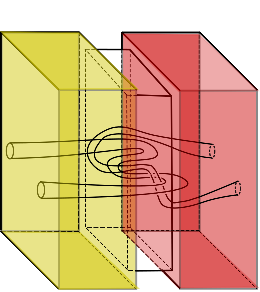
\end{array}
\ar[r] &
\begin{array}{c}
\def\svgwidth{1.25in}
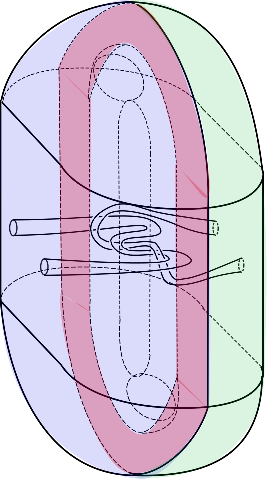
\end{array}
}
}
\]
\caption{The identifications of the bottom cube in Figure \ref{fig_gabai_cubes} to obtain $N_{J}$.} \label{fig_gabai_cubes_identifications}
\end{figure}

\begin{figure}
\begin{center}
\fcolorbox{black}{white}{
\scalebox{1}{\psfig{figure=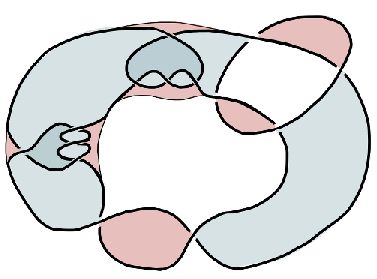}}}
\end{center}
\caption{Gabai's fiber for $N_{\tilde{J}}$.} \label{fig_gabai_fiber} 
\end{figure}

Let $L=J \sqcup K$ and suppose that $J$ is a virtually fibered link admitting a finite-index covering $p:N_{\tilde{J}} \to N_J$, where $\tilde{J}$ is a fibered link. Denote the index by $n$. Assume also that $p$ is regular and orientation preserving. Suppose that there is a knot $\tilde{K}$ in $N_{\tilde{J}}$ such that $p(\tilde{K})=K$ (as functions defined on $\mathbb{S}^1$). By regularity, there are $n$ such knots $\tilde{K}=\tilde{K}(1),\ldots,\tilde{K}(n)$. Suppose that each $\tilde{K}(i)$ is in SSF with respect to some fiber $\Sigma(i)$ of $\tilde{J}$. The \emph{associated virtual knot spectrum} of $K$ in $N_J$ is the formal sum:
\[
\sigma_K=\sum_{i=1}^n \kappa([\tilde{K}(i);\Sigma(i)]).
\]
We consider $\sigma_K$ as an element of $\mathbb{Z}[\mathscr{V}]$, the free abelian group generated by the equivalence classes of oriented virtual knots.

\begin{theorem} \label{theorem_spectrum} Let $K_1^{N_{J}}$, $K_2^{N_J}$ have associated virtual knot spectra $\sigma_1$, $\sigma_2$, respectively. If $K_1^{N_J} \leftrightharpoons K_2^{N_J}$, then $\sigma_1=\sigma_2$.
\end{theorem}
\begin{proof} View the ambient isotopy as an isotopy of the functions $K_i:\mathbb{S}^1 \to N_J$. The isotopy lifts to $N_{\tilde{J}}$ and connects some lift of $K_1$ to some lift of $K_2$. Again applying the isotopy extension theorem we see that these are equivalences of knots in $N_{\tilde{J}}$. Say $\tilde{K}_1(i) \leftrightharpoons \tilde{K}_2(j_i)$. Apply this to all lifts of $K_1$ to get a permutation of $1,\ldots,n$ defined by  $i \to j_i$.
\newline
\newline
Let $\tilde{L}_k(i)=\tilde{J} \sqcup \tilde{K}_k(i)$. By the above, we have an equivalence $\tilde{L}_1(i) \leftrightharpoons \tilde{L}_2(j_i)$. By hypothesis, each $\tilde{K}_k(i)$ is in SSF with respect to some fiber $\Sigma_k(i)$. By Theorem \ref{theorem_link_invar}, it follows that $\kappa([\tilde{K}_1(i);\Sigma_1(i)]) \leftrightharpoons \kappa([\tilde{K}_2(j_i);\Sigma_2(j_i)])$. Hence $\sigma_1=\sigma_2$. 
\end{proof}


\textbf{Example:} Let $J$, $\tilde{J}$ be as in Gabai's example. The covering $p:N_{\tilde{J}} \to N_J$ is regular because it has index 2. Clearly $p$ preserves orientation. Let $K$ be the (red) knot drawn in the detail of the middle panel of Figure \ref{fig_gabai_non_invert} oriented arbitrarily. We will show $K$ is non-invertible in $N_J$. $K$ has two lifts $\tilde{K}(1)$ and $\tilde{K}(2)$. $\tilde{K}(1)$ in SSF with respect to Gabai's fiber $\Sigma_{\tilde{J}}$ (top panel). The invariant associated virtual knot $\upsilon(1)=\kappa([\tilde{K}(1);\Sigma_J])$ is in the bottom left panel. This is equivalent to the non-invertible classical knot $8_{17}$, oriented as $\vec{8}_{17}$. The knot $\tilde{K}(2)$ can also be put into special SSF with respect to $\Sigma_J$ (after some effort). The associated virtual knot $\upsilon(2)$ is also $\vec{8}_{17}$. Hence $\sigma_K=2\cdot \vec{8}_{17}$.  But $\sigma_{-K}=2 \cdot (-\vec{8}_{17})$. Hence $K \not \leftrightharpoons -K$ in $N_J$. Note however that $K$ is a trefoil in $\mathbb{S}^3$, and hence is invertible in $\mathbb{S}^3$ (bottom right panel). \hfill $\square$

\begin{figure}[htb]
\fcolorbox{black}{white}{
\begin{tabular}{c} \\ \\
\begin{tabular}{c}
\fcolorbox{black}{white}{
\def\svgwidth{4.7in}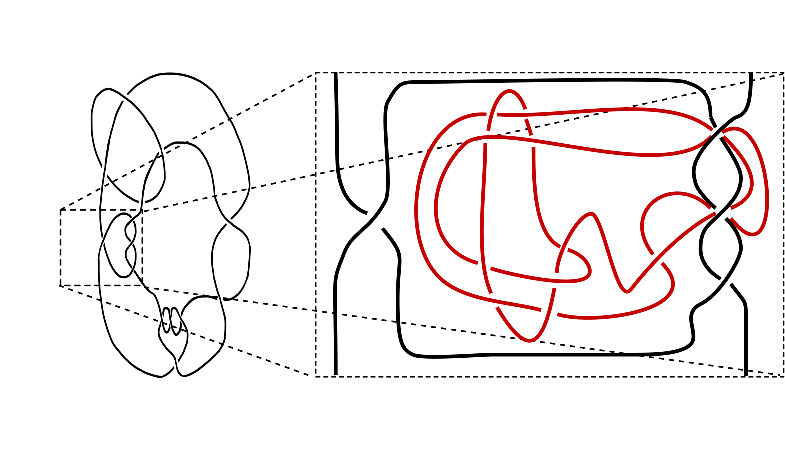} \\ \\
\fcolorbox{black}{white}{
\def\svgwidth{4.7in}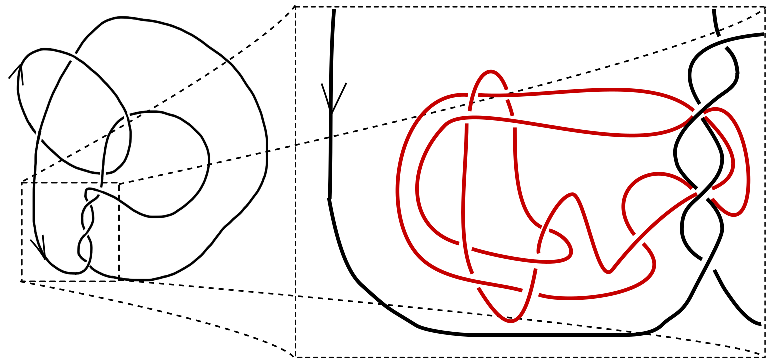}
\end{tabular} \\ \\
\begin{tabular}{ccc}
\fcolorbox{black}{white}{\def\svgwidth{2.3in}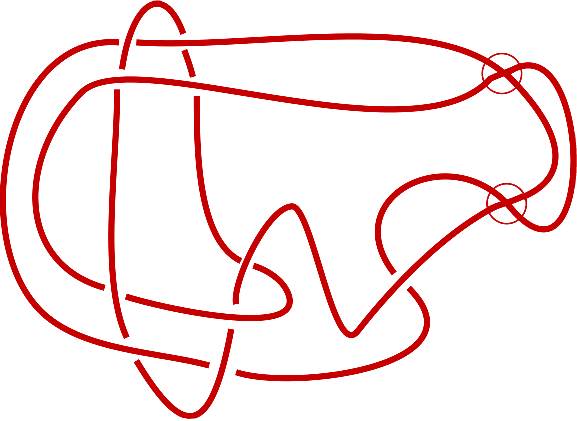} & & \fcolorbox{black}{white}{\def\svgwidth{2.3in}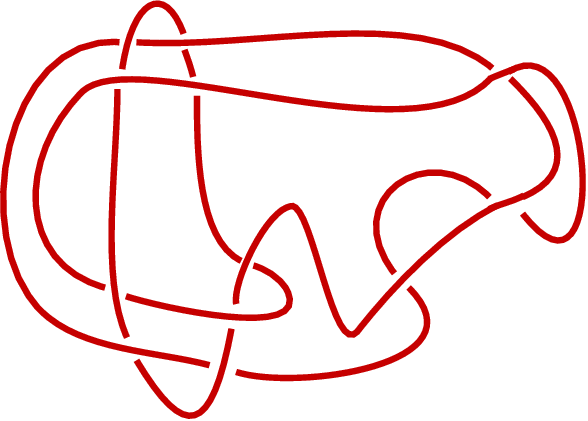}
\end{tabular} \\ \\
\end{tabular}
}
\caption{The knot $K$ in $N_J$ (middle, red), the knot $\tilde{K}(1)$ in $N_{\tilde{J}}$ (top, red), associated virtual knot $\upsilon(1) \leftrightharpoons 8_{17}$ (bottom left), and $K$ depicted as a left handed trefoil in $\mathbb{S}^3$ (bottom right).}\label{fig_gabai_non_invert}
\end{figure}

\subsection{Thanks} To A. Kaestner, K. Orr, and R. Todd, for encouragement. 

\bibliographystyle{plain}
\bibliography{bib_fiber}

\end{document}

%% file: ssf_cross.eps_tex
\begingroup%
  \makeatletter%
  \providecommand\color[2][]{%
    \errmessage{(Inkscape) Color is used for the text in Inkscape, but the package 'color.sty' is not loaded}%
    \renewcommand\color[2][]{}%
  }%
  \providecommand\transparent[1]{%
    \errmessage{(Inkscape) Transparency is used (non-zero) for the text in Inkscape, but the package 'transparent.sty' is not loaded}%
    \renewcommand\transparent[1]{}%
  }%
  \providecommand\rotatebox[2]{#2}%
  \ifx\svgwidth\undefined%
    \setlength{\unitlength}{201.55661621bp}%
    \ifx\svgscale\undefined%
      \relax%
    \else%
      \setlength{\unitlength}{\unitlength * \real{\svgscale}}%
    \fi%
  \else%
    \setlength{\unitlength}{\svgwidth}%
  \fi%
  \global\let\svgwidth\undefined%
  \global\let\svgscale\undefined%
  \makeatother%
  \begin{picture}(1,0.45961753)%
    \put(0,0){\includegraphics[width=\unitlength]{ssf_cross.eps}}%
    \put(-0.004942,0.033286){\color[rgb]{0,0,0}\makebox(0,0)[lb]{\smash{$\Sigma_J$}}}%
    \put(0.58963355,0.41438366){\color[rgb]{0,0,0}\makebox(0,0)[lb]{\smash{$B_i$}}}%
    \put(0.890084,0.1708607){\color[rgb]{0,0,0}\makebox(0,0)[lb]{\smash{$K$}}}%
  \end{picture}%
\endgroup%

%% file: virt_cov_ex_1_II.eps_tex
\begingroup%
  \makeatletter%
  \providecommand\color[2][]{%
    \errmessage{(Inkscape) Color is used for the text in Inkscape, but the package 'color.sty' is not loaded}%
    \renewcommand\color[2][]{}%
  }%
  \providecommand\transparent[1]{%
    \errmessage{(Inkscape) Transparency is used (non-zero) for the text in Inkscape, but the package 'transparent.sty' is not loaded}%
    \renewcommand\transparent[1]{}%
  }%
  \providecommand\rotatebox[2]{#2}%
  \ifx\svgwidth\undefined%
    \setlength{\unitlength}{120.9376121bp}%
    \ifx\svgscale\undefined%
      \relax%
    \else%
      \setlength{\unitlength}{\unitlength * \real{\svgscale}}%
    \fi%
  \else%
    \setlength{\unitlength}{\svgwidth}%
  \fi%
  \global\let\svgwidth\undefined%
  \global\let\svgscale\undefined%
  \makeatother%
  \begin{picture}(1,1.10016862)%
    \put(0,0){\includegraphics[width=\unitlength]{virt_cov_ex_1_II.eps}}%
    \put(0.81275785,0.01986432){\color[rgb]{0,0,0}\makebox(0,0)[lb]{\smash{$J$}}}%
    \put(0.63,0.23){\color[rgb]{0,0,0}\makebox(0,0)[lb]{\smash{$K$}}}%
  \end{picture}%
\endgroup%

%% file: virt_cov_ex_2_II.eps_tex
\begingroup%
  \makeatletter%
  \providecommand\color[2][]{%
    \errmessage{(Inkscape) Color is used for the text in Inkscape, but the package 'color.sty' is not loaded}%
    \renewcommand\color[2][]{}%
  }%
  \providecommand\transparent[1]{%
    \errmessage{(Inkscape) Transparency is used (non-zero) for the text in Inkscape, but the package 'transparent.sty' is not loaded}%
    \renewcommand\transparent[1]{}%
  }%
  \providecommand\rotatebox[2]{#2}%
  \ifx\svgwidth\undefined%
    \setlength{\unitlength}{104.66176344bp}%
    \ifx\svgscale\undefined%
      \relax%
    \else%
      \setlength{\unitlength}{\unitlength * \real{\svgscale}}%
    \fi%
  \else%
    \setlength{\unitlength}{\svgwidth}%
  \fi%
  \global\let\svgwidth\undefined%
  \global\let\svgscale\undefined%
  \makeatother%
  \begin{picture}(1,1.14084052)%
    \put(0,0){\includegraphics[width=\unitlength]{virt_cov_ex_2_II.eps}}%
    \put(-0.00951727,0.81228389){\color[rgb]{0,0,0}\makebox(0,0)[lb]{\smash{$\upsilon$}}}%
  \end{picture}%
\endgroup%

%% file: hopf_link_example.eps_tex
\begingroup%
  \makeatletter%
  \providecommand\color[2][]{%
    \errmessage{(Inkscape) Color is used for the text in Inkscape, but the package 'color.sty' is not loaded}%
    \renewcommand\color[2][]{}%
  }%
  \providecommand\transparent[1]{%
    \errmessage{(Inkscape) Transparency is used (non-zero) for the text in Inkscape, but the package 'transparent.sty' is not loaded}%
    \renewcommand\transparent[1]{}%
  }%
  \providecommand\rotatebox[2]{#2}%
  \ifx\svgwidth\undefined%
    \setlength{\unitlength}{83.20661289bp}%
    \ifx\svgscale\undefined%
      \relax%
    \else%
      \setlength{\unitlength}{\unitlength * \real{\svgscale}}%
    \fi%
  \else%
    \setlength{\unitlength}{\svgwidth}%
  \fi%
  \global\let\svgwidth\undefined%
  \global\let\svgscale\undefined%
  \makeatother%
  \begin{picture}(1,0.90825403)%
    \put(0,0){\includegraphics[width=\unitlength]{hopf_link_example.eps}}%
  \end{picture}%
\endgroup%

%% file: hopf_link_example_on_seifert.eps_tex
\begingroup%
  \makeatletter%
  \providecommand\color[2][]{%
    \errmessage{(Inkscape) Color is used for the text in Inkscape, but the package 'color.sty' is not loaded}%
    \renewcommand\color[2][]{}%
  }%
  \providecommand\transparent[1]{%
    \errmessage{(Inkscape) Transparency is used (non-zero) for the text in Inkscape, but the package 'transparent.sty' is not loaded}%
    \renewcommand\transparent[1]{}%
  }%
  \providecommand\rotatebox[2]{#2}%
  \ifx\svgwidth\undefined%
    \setlength{\unitlength}{82.5878506bp}%
    \ifx\svgscale\undefined%
      \relax%
    \else%
      \setlength{\unitlength}{\unitlength * \real{\svgscale}}%
    \fi%
  \else%
    \setlength{\unitlength}{\svgwidth}%
  \fi%
  \global\let\svgwidth\undefined%
  \global\let\svgscale\undefined%
  \makeatother%
  \begin{picture}(1,0.56237748)%
    \put(0,0){\includegraphics[width=\unitlength]{hopf_link_example_on_seifert.eps}}%
  \end{picture}%
\endgroup%

%% file: hopf_link_example_virt.eps_tex
\begingroup%
  \makeatletter%
  \providecommand\color[2][]{%
    \errmessage{(Inkscape) Color is used for the text in Inkscape, but the package 'color.sty' is not loaded}%
    \renewcommand\color[2][]{}%
  }%
  \providecommand\transparent[1]{%
    \errmessage{(Inkscape) Transparency is used (non-zero) for the text in Inkscape, but the package 'transparent.sty' is not loaded}%
    \renewcommand\transparent[1]{}%
  }%
  \providecommand\rotatebox[2]{#2}%
  \ifx\svgwidth\undefined%
    \setlength{\unitlength}{55.05500644bp}%
    \ifx\svgscale\undefined%
      \relax%
    \else%
      \setlength{\unitlength}{\unitlength * \real{\svgscale}}%
    \fi%
  \else%
    \setlength{\unitlength}{\svgwidth}%
  \fi%
  \global\let\svgwidth\undefined%
  \global\let\svgscale\undefined%
  \makeatother%
  \begin{picture}(1,0.77925199)%
    \put(0,0){\includegraphics[width=\unitlength]{hopf_link_example_virt.eps}}%
  \end{picture}%
\endgroup%

%% file: non_invert_example_II.eps_tex
\begingroup%
  \makeatletter%
  \providecommand\color[2][]{%
    \errmessage{(Inkscape) Color is used for the text in Inkscape, but the package 'color.sty' is not loaded}%
    \renewcommand\color[2][]{}%
  }%
  \providecommand\transparent[1]{%
    \errmessage{(Inkscape) Transparency is used (non-zero) for the text in Inkscape, but the package 'transparent.sty' is not loaded}%
    \renewcommand\transparent[1]{}%
  }%
  \providecommand\rotatebox[2]{#2}%
  \ifx\svgwidth\undefined%
    \setlength{\unitlength}{394.44990032bp}%
    \ifx\svgscale\undefined%
      \relax%
    \else%
      \setlength{\unitlength}{\unitlength * \real{\svgscale}}%
    \fi%
  \else%
    \setlength{\unitlength}{\svgwidth}%
  \fi%
  \global\let\svgwidth\undefined%
  \global\let\svgscale\undefined%
  \makeatother%
  \begin{picture}(1,0.64179276)%
    \put(0,0){\includegraphics[width=\unitlength]{non_invert_example_II.eps}}%
    \put(0.33770724,0.12971173){\color[rgb]{0,0,0}\makebox(0,0)[lb]{\smash{$K$}}}%
    \put(-0.00252527,0.05095421){\color[rgb]{0,0,0}\makebox(0,0)[lb]{\smash{$J$}}}%
  \end{picture}%
\endgroup%

%% file: non_invert_example_gauss.eps_tex
\begingroup%
  \makeatletter%
  \providecommand\color[2][]{%
    \errmessage{(Inkscape) Color is used for the text in Inkscape, but the package 'color.sty' is not loaded}%
    \renewcommand\color[2][]{}%
  }%
  \providecommand\transparent[1]{%
    \errmessage{(Inkscape) Transparency is used (non-zero) for the text in Inkscape, but the package 'transparent.sty' is not loaded}%
    \renewcommand\transparent[1]{}%
  }%
  \providecommand\rotatebox[2]{#2}%
  \ifx\svgwidth\undefined%
    \setlength{\unitlength}{206.43890275bp}%
    \ifx\svgscale\undefined%
      \relax%
    \else%
      \setlength{\unitlength}{\unitlength * \real{\svgscale}}%
    \fi%
  \else%
    \setlength{\unitlength}{\svgwidth}%
  \fi%
  \global\let\svgwidth\undefined%
  \global\let\svgscale\undefined%
  \makeatother%
  \begin{picture}(1,0.9738959)%
    \put(0,0){\includegraphics[width=\unitlength]{non_invert_example_gauss.eps}}%
  \end{picture}%
\endgroup%

%% file: whitten_theorem_almost.eps_tex
\begingroup%
  \makeatletter%
  \providecommand\color[2][]{%
    \errmessage{(Inkscape) Color is used for the text in Inkscape, but the package 'color.sty' is not loaded}%
    \renewcommand\color[2][]{}%
  }%
  \providecommand\transparent[1]{%
    \errmessage{(Inkscape) Transparency is used (non-zero) for the text in Inkscape, but the package 'transparent.sty' is not loaded}%
    \renewcommand\transparent[1]{}%
  }%
  \providecommand\rotatebox[2]{#2}%
  \ifx\svgwidth\undefined%
    \setlength{\unitlength}{367.0789318bp}%
    \ifx\svgscale\undefined%
      \relax%
    \else%
      \setlength{\unitlength}{\unitlength * \real{\svgscale}}%
    \fi%
  \else%
    \setlength{\unitlength}{\svgwidth}%
  \fi%
  \global\let\svgwidth\undefined%
  \global\let\svgscale\undefined%
  \makeatother%
  \begin{picture}(.82,0.24648567)%
    \put(-.25,0){\includegraphics[width=\unitlength]{whitten_theorem_almost.eps}}%
    \put(0.74609918,0.05880749){\color[rgb]{0,0,0}\makebox(0,0)[lb]{\smash{$\Sigma_H$}}}%
    \put(0.2976845,0.05344841){\color[rgb]{0,0,0}\makebox(0,0)[lb]{\smash{$\Sigma_{2,1}$}}}%
    \put(0.74609918,0.17849411){\color[rgb]{0,0,0}\makebox(0,0)[lb]{\smash{$\Sigma_T$}}}%
  \end{picture}%
\endgroup%

%% file: l_tilde_2.eps_tex
\begingroup%
  \makeatletter%
  \providecommand\color[2][]{%
    \errmessage{(Inkscape) Color is used for the text in Inkscape, but the package 'color.sty' is not loaded}%
    \renewcommand\color[2][]{}%
  }%
  \providecommand\transparent[1]{%
    \errmessage{(Inkscape) Transparency is used (non-zero) for the text in Inkscape, but the package 'transparent.sty' is not loaded}%
    \renewcommand\transparent[1]{}%
  }%
  \providecommand\rotatebox[2]{#2}%
  \ifx\svgwidth\undefined%
    \setlength{\unitlength}{313bp}%
    \ifx\svgscale\undefined%
      \relax%
    \else%
      \setlength{\unitlength}{\unitlength * \real{\svgscale}}%
    \fi%
  \else%
    \setlength{\unitlength}{\svgwidth}%
  \fi%
  \global\let\svgwidth\undefined%
  \global\let\svgscale\undefined%
  \makeatother%
  \begin{picture}(1,0.49840256)%
    \put(0,0){\includegraphics[width=\unitlength]{l_tilde_2.eps}}%
  \end{picture}%
\endgroup%

%% file: l_non_fibered_2.eps_tex
\begingroup%
  \makeatletter%
  \providecommand\color[2][]{%
    \errmessage{(Inkscape) Color is used for the text in Inkscape, but the package 'color.sty' is not loaded}%
    \renewcommand\color[2][]{}%
  }%
  \providecommand\transparent[1]{%
    \errmessage{(Inkscape) Transparency is used (non-zero) for the text in Inkscape, but the package 'transparent.sty' is not loaded}%
    \renewcommand\transparent[1]{}%
  }%
  \providecommand\rotatebox[2]{#2}%
  \ifx\svgwidth\undefined%
    \setlength{\unitlength}{313bp}%
    \ifx\svgscale\undefined%
      \relax%
    \else%
      \setlength{\unitlength}{\unitlength * \real{\svgscale}}%
    \fi%
  \else%
    \setlength{\unitlength}{\svgwidth}%
  \fi%
  \global\let\svgwidth\undefined%
  \global\let\svgscale\undefined%
  \makeatother%
  \begin{picture}(1,0.74121406)%
    \put(0,0){\includegraphics[width=\unitlength]{l_non_fibered_2.eps}}%
  \end{picture}%
\endgroup%

%% file: gabai_fun_region_identify_II.eps_tex
\begingroup%
  \makeatletter%
  \providecommand\color[2][]{%
    \errmessage{(Inkscape) Color is used for the text in Inkscape, but the package 'color.sty' is not loaded}%
    \renewcommand\color[2][]{}%
  }%
  \providecommand\transparent[1]{%
    \errmessage{(Inkscape) Transparency is used (non-zero) for the text in Inkscape, but the package 'transparent.sty' is not loaded}%
    \renewcommand\transparent[1]{}%
  }%
  \providecommand\rotatebox[2]{#2}%
  \ifx\svgwidth\undefined%
    \setlength{\unitlength}{124.68408517bp}%
    \ifx\svgscale\undefined%
      \relax%
    \else%
      \setlength{\unitlength}{\unitlength * \real{\svgscale}}%
    \fi%
  \else%
    \setlength{\unitlength}{\svgwidth}%
  \fi%
  \global\let\svgwidth\undefined%
  \global\let\svgscale\undefined%
  \makeatother%
  \begin{picture}(1,1.11139208)%
    \put(0,0){\includegraphics[width=\unitlength]{gabai_fun_region_identify_II.eps}}%
    \put(0.04595598,1.03826978){\color[rgb]{0,0,0}\makebox(0,0)[lb]{\smash{$Y$}}}%
    \put(0.52645737,1.03223837){\color[rgb]{0,0,0}\makebox(0,0)[lb]{\smash{$R$}}}%
  \end{picture}%
\endgroup%

%% file: gabai_fun_region_glued_II.eps_tex
\begingroup%
  \makeatletter%
  \providecommand\color[2][]{%
    \errmessage{(Inkscape) Color is used for the text in Inkscape, but the package 'color.sty' is not loaded}%
    \renewcommand\color[2][]{}%
  }%
  \providecommand\transparent[1]{%
    \errmessage{(Inkscape) Transparency is used (non-zero) for the text in Inkscape, but the package 'transparent.sty' is not loaded}%
    \renewcommand\transparent[1]{}%
  }%
  \providecommand\rotatebox[2]{#2}%
  \ifx\svgwidth\undefined%
    \setlength{\unitlength}{130.03113433bp}%
    \ifx\svgscale\undefined%
      \relax%
    \else%
      \setlength{\unitlength}{\unitlength * \real{\svgscale}}%
    \fi%
  \else%
    \setlength{\unitlength}{\svgwidth}%
  \fi%
  \global\let\svgwidth\undefined%
  \global\let\svgscale\undefined%
  \makeatother%
  \begin{picture}(1,1.76818849)%
    \put(0,0){\includegraphics[width=\unitlength]{gabai_fun_region_glued_II.eps}}%
    \put(0.03667572,0.0464975){\color[rgb]{0,0,0}\makebox(0,0)[lb]{\smash{$B$}}}%
    \put(0.81862817,1.66843788){\color[rgb]{0,0,0}\makebox(0,0)[lb]{\smash{$G$}}}%
  \end{picture}%
\endgroup%

%% file: gabai_non_invert_III.eps_tex
\begingroup%
  \makeatletter%
  \providecommand\color[2][]{%
    \errmessage{(Inkscape) Color is used for the text in Inkscape, but the package 'color.sty' is not loaded}%
    \renewcommand\color[2][]{}%
  }%
  \providecommand\transparent[1]{%
    \errmessage{(Inkscape) Transparency is used (non-zero) for the text in Inkscape, but the package 'transparent.sty' is not loaded}%
    \renewcommand\transparent[1]{}%
  }%
  \providecommand\rotatebox[2]{#2}%
  \ifx\svgwidth\undefined%
    \setlength{\unitlength}{376.51016806bp}%
    \ifx\svgscale\undefined%
      \relax%
    \else%
      \setlength{\unitlength}{\unitlength * \real{\svgscale}}%
    \fi%
  \else%
    \setlength{\unitlength}{\svgwidth}%
  \fi%
  \global\let\svgwidth\undefined%
  \global\let\svgscale\undefined%
  \makeatother%
  \begin{picture}(1,0.57369234)%
    \put(0,0){\includegraphics[width=\unitlength]{gabai_non_invert_III.eps}}%
    \put(0.68746546,0.33696219){\color[rgb]{0,0,0}\makebox(0,0)[lb]{\smash{$\tilde{K}(1)$}}}%
  \end{picture}%
\endgroup%

%% file: gabai_non_invert_in_mani_II.eps_tex
\begingroup%
  \makeatletter%
  \providecommand\color[2][]{%
    \errmessage{(Inkscape) Color is used for the text in Inkscape, but the package 'color.sty' is not loaded}%
    \renewcommand\color[2][]{}%
  }%
  \providecommand\transparent[1]{%
    \errmessage{(Inkscape) Transparency is used (non-zero) for the text in Inkscape, but the package 'transparent.sty' is not loaded}%
    \renewcommand\transparent[1]{}%
  }%
  \providecommand\rotatebox[2]{#2}%
  \ifx\svgwidth\undefined%
    \setlength{\unitlength}{367.94846513bp}%
    \ifx\svgscale\undefined%
      \relax%
    \else%
      \setlength{\unitlength}{\unitlength * \real{\svgscale}}%
    \fi%
  \else%
    \setlength{\unitlength}{\svgwidth}%
  \fi%
  \global\let\svgwidth\undefined%
  \global\let\svgscale\undefined%
  \makeatother%
  \begin{picture}(1,0.46592806)%
    \put(0,0){\includegraphics[width=\unitlength]{gabai_non_invert_in_mani_II.eps}}%
    \put(0.70581031,0.24725448){\color[rgb]{0,0,0}\makebox(0,0)[lb]{\smash{$K$}}}%
  \end{picture}%
\endgroup%

%% file: gabai_non_invert_virtual_cover.eps_tex
\begingroup%
  \makeatletter%
  \providecommand\color[2][]{%
    \errmessage{(Inkscape) Color is used for the text in Inkscape, but the package 'color.sty' is not loaded}%
    \renewcommand\color[2][]{}%
  }%
  \providecommand\transparent[1]{%
    \errmessage{(Inkscape) Transparency is used (non-zero) for the text in Inkscape, but the package 'transparent.sty' is not loaded}%
    \renewcommand\transparent[1]{}%
  }%
  \providecommand\rotatebox[2]{#2}%
  \ifx\svgwidth\undefined%
    \setlength{\unitlength}{276.79242606bp}%
    \ifx\svgscale\undefined%
      \relax%
    \else%
      \setlength{\unitlength}{\unitlength * \real{\svgscale}}%
    \fi%
  \else%
    \setlength{\unitlength}{\svgwidth}%
  \fi%
  \global\let\svgwidth\undefined%
  \global\let\svgscale\undefined%
  \makeatother%
  \begin{picture}(1,0.72542752)%
    \put(0,0){\includegraphics[width=\unitlength]{gabai_non_invert_virtual_cover.eps}}%
  \end{picture}%
\endgroup%

%% file: gabai_non_invert_in_S3.eps_tex
\begingroup%
  \makeatletter%
  \providecommand\color[2][]{%
    \errmessage{(Inkscape) Color is used for the text in Inkscape, but the package 'color.sty' is not loaded}%
    \renewcommand\color[2][]{}%
  }%
  \providecommand\transparent[1]{%
    \errmessage{(Inkscape) Transparency is used (non-zero) for the text in Inkscape, but the package 'transparent.sty' is not loaded}%
    \renewcommand\transparent[1]{}%
  }%
  \providecommand\rotatebox[2]{#2}%
  \ifx\svgwidth\undefined%
    \setlength{\unitlength}{281.19350603bp}%
    \ifx\svgscale\undefined%
      \relax%
    \else%
      \setlength{\unitlength}{\unitlength * \real{\svgscale}}%
    \fi%
  \else%
    \setlength{\unitlength}{\svgwidth}%
  \fi%
  \global\let\svgwidth\undefined%
  \global\let\svgscale\undefined%
  \makeatother%
  \begin{picture}(1,0.71407354)%
    \put(0,0){\includegraphics[width=\unitlength]{gabai_non_invert_in_S3.eps}}%
  \end{picture}%
\endgroup%